\documentclass{jcmlatex}

\usepackage[letterpaper,margin=1in]{geometry}
\usepackage{moreverb}
\usepackage{amssymb,amsmath}
\usepackage{graphicx}

\usepackage{enumerate}

\usepackage{setspace}

\usepackage[colorlinks,bookmarksopen,bookmarksnumbered,linkcolor=blue,citecolor=blue,urlcolor=blue]{hyperref}

\usepackage{algorithmic,algorithm}

\def\JCMvol{xx}
\def\JCMno{x}
\def\JCMyear{200x}
\def\JCMreceived{xxx}
\def\JCMrevised{xxx}
\def\JCMaccepted{xxx}

\newtheorem{Theorem}{Theorem}[section]
\newtheorem{Lemma}[Theorem]{Lemma}
\newtheorem{Proposition}[Theorem]{Proposition}
\newtheorem{Corollary}[Theorem]{Corollary}
\newtheorem{Definition}[Theorem]{Definition}
\newtheorem{Remark}[Theorem]{Remark}

\newcommand{\R}{\mathbb{R}}
\newcommand{\Z}{\mathbb{Z}}

\newcommand{\blm}{\begin{Lemma}}
\newcommand{\elm}{\end{Lemma}}
\newcommand{\bthm}{\begin{Theorem}}
\newcommand{\ethm}{\end{Theorem}}
\newcommand{\bcor}{\begin{Corollary}}
\newcommand{\ecor}{\end{Corollary}}
\newcommand{\bdf}{\begin{Definition}}
\newcommand{\edf}{\end{Definition}}

 \newcommand{\bone}{\ensuremath{\mathbf {1}}}

\begin{document}

\markboth{J. C. URSCHEL, X. HU, J. XU, L. T. ZIKATANOV}{A Cascadic Multigrid Algorithm for Computing the Fiedler Vector of Graph Laplacians}

\title{A CASCADIC MULTIGRID ALGORITHM FOR COMPUTING \\ THE FIEDLER VECTOR OF GRAPH LAPLACIANS}

\author{John C. Urschel
\thanks{Department of Mathematics, Penn State University, Pennsylvania, USA \\ Email: urschel@math.psu.edu}
\and
Xiaozhe Hu
\thanks{Department of Mathematics, Tufts University, Massachusetts, USA \\ Email: Xiaozhe.Hu@tufts.edu}
\and
Jinchao Xu
\thanks{Department of Mathematics, Penn State University, Pennsylvania, USA \\ Email: xu@math.psu.edu}
\and
Ludmil T. Zikatanov
\thanks{Department of Mathematics, Penn State University, Pennsylvania, USA \\ 
Institute of Mathematics and Informatics, Bulgarian Academy of Sciences, Sofia, Bulgaria  \\ 
Email: ludmil@psu.edu}}

\maketitle

\begin{abstract}
In this paper, we develop a cascadic multigrid algorithm for fast computation of the Fiedler vector of a graph Laplacian, namely, the eigenvector corresponding to the second smallest eigenvalue. This vector has been found to have applications in fields such as graph partitioning and graph drawing. The algorithm is a purely algebraic approach based on a heavy edge coarsening scheme and pointwise smoothing for refinement.  To gain theoretical insight, we also consider the related cascadic multigrid method in the geometric setting for elliptic eigenvalue problems and show its uniform convergence under certain assumptions. Numerical tests are presented for computing the Fiedler vector of several practical graphs, and numerical results show the efficiency and optimality of our proposed cascadic multigrid algorithm. 
\end{abstract}



\section{Introduction}
Computation of the Fiedler vector of graph Laplacians has proven to be a relevant topic, and has found applications in areas such as graph partitioning and graph drawing \cite{KCH}. There have been a number of techniques implemented for computation of the Fiedler vector, most notably by Barnard and Simon \cite{BS}. They implemented a multilevel coarsening strategy, using maximal independent sets and created a matching from them. For the refinement procedure, Rayleigh quotient iteration was used. We note that the term refinement refers to the smoothing process that occurs, and has a different meaning in the multigrid literature. Although at the time this was significantly faster than the standard recursive spectral bisection, it leaves room for improvement. The majority of the improvement has been in the form of coarsening algorithms. Better coarsening techniques, such as heavy edge matching (HEM), have been used more frequently, and have exhibited much shorter run times \cite{KK1,KK2}. 

For more general eigenproblems of symmetric positive definite matrices, techniques such as Jacobi-Davison \cite{SV} and the Locally Optimal Preconditioned Conjugate Gradient Method~\cite{KN} (see also~\cite{LAOK}) have been used and shown to give good approximations to eigenvalues and eigenvectors. These techniques can easily be extended to computing a Fiedler vector.  Other eigensolvers are provided by setting an Algebraic MultiGrid (AMG) tuned specifically for graph Laplacians (see, e.g. Lean AMG~\cite{LB}) as a preconditioner in the LOPCG Method. 

In this paper, we introduce a new and fast coarsening algorithm, based on the conecpt of heavy edge matching, with a more aggressive coarsening procedure. For refinement, we implement a form of power iteration. For both our coarsening and refinement procedures we have created algorithms that are straightforward  to implement. While heavy edge matching is complicated and tough to implement in high level programming languages, since it involves selecting an edge with heaviest weight between two unmatched vertices, heavy edge coarsening is significantly easier because we do not need to worry about whether a vertex has been aggregated or not.  For the refinement procedure, power iteration does not require the inversion of a matrix, making its use much more straightforward than for Rayleigh quotient iteration, which requires some technique to approximately invert the matrix.  

Based on these two improved components, we propose a cascadic multigrid (CMG) method to compute the Fiedler vector. The CMG method has been treated in the literature, most notably by Bornemann and Deuflhard \cite{BD1,BD2}, Braess, Deuflhard, and Lipnikov \cite{BDL}, and Shaidurov \cite{S1,S2,S3}. However, little has been done with respect to the elliptic eigenvalue problem. Our technique is a purely algebraic approach which only uses the given graph.  Moreover, although the purely algebraic approach is technically difficult to analyze, we consider the CMG method for the elliptic eigenvalue problem in the geometric setting.  Based on the standard smoothing property and approximation property, we show that the geometric CMG method converges uniformly for the model problem, which indirectly provides theoretical justification of the efficiency of the CMG method.  This also shows the potential of our CMG method for solving other eigenvalue problems from different applications.  

The remainder of the paper is organized as follows.  In Section \ref{sec:CMG-Fiedler}, we briefly review the Fiedler vector and introduce our cascadic multigrid method for computing the Fiedler vector of a graph Laplacian.  The cascadic multigrid method for elliptic eigenvalue problems is proposed in Section \ref{sec:GCMG} and its convergence analysis is also provided.  Section \ref{numerics} presents numerical experiments to support the theoretical results of CMG method for elliptic eigenvalue problems and demonstrate its efficiency for computing the Fiedler vector of some graph Laplacian problems from real applications.  We conclude the paper in Section \ref{sec:conclusion} by some general remarks on this work and proposed future work. 

\section{Cascadic MG Method for Computing the Fiedler Vector} \label{sec:CMG-Fiedler}

We begin by formally introducing the concept of a graph Laplacian and Fiedler vector. We start with the concept of a graph. A weighted graph $G=(V,E,w)$ is said to be undirected if the edges have no orientation. A graph is a multigraph if $(i,i) \notin E$ for all $1 \le i \le \vert V \vert$ ($\vert V \vert$ is the number of vertices).  For the remainder of this paper, we assume that all graphs are undirected and multigraphs. 

We consider the task of representing a graph in matrix form. One of the most natural representations is through its Laplacian. The Laplacian of a graph is defined as follows: 
\bdf
Let $G=(V,E,w)$ be a weighted graph. We define the Laplacian matrix of $G$, denoted $L(G) \in \R^{n \times n}$ (or just $L$ for short), $n= \vert V \vert$, as follows:
\begin{equation*}
L(G)_{(i,j)}:= \biggl \{  
\begin{matrix}
\>  \> \> d_{v_i}, \; \> \quad \mbox{for}  \ \quad i=j, \\ 
- w_{i,j}, \quad \mbox{for} \  \quad i \neq j,
\end{matrix}
\end{equation*}
where $d_{v_i}$ is the degree of $v_i$, and $w_{i,j}$ is the weight of the edge connecting $v_i$ and $v_j$. 
\edf

The Laplacian $L(G)$ is self-adjoint, positive semi-definite, and diagonally dominant. In addition, the sum of any row (and also, any column) of $L$ is zero. Therefore $\lambda = 0$ is an eigenvalue of $L$, with corresponding eigenvector $\bone =(1,...,1)^{T}$.  Let us order the eigenvalues of $L(G)$ as follows: $0= \lambda_1 \le \lambda_2 \le ... \le \lambda_n$, and denote by $\varphi_1, \varphi_2, ..., \varphi_n$ the corresponding eigenvectors. We have already seen that $\varphi_1 = \alpha \bone$. We now consider $\lambda_2$ and $\varphi_2$. This eigenvalue and eigenvector pair has special significance and, for this reason, are given special names.

\bdf The algebraic connectivity of a graph $G$, denoted by $a(G)$, is defined to be the second smallest eigenvalue of the corresponding Laplacian matrix $L(G)$, with eigenvalues $0= \lambda_1 \le \lambda_2 \le ... \le \lambda_n$ and eigenvectors $\varphi_1, \varphi_2, ..., \varphi_n$. The eigenvector $\varphi_2$, corresponding to the eigenvalue $a(G)$, is called the Fiedler vector of $G$.  \edf

The term Fiedler vector comes from the mathematician Miroslav Fiedler, who proved many results regarding the significance of this eigenvector. His work involving irreducible matrices and the Fiedler vector can be found in \cite{F1,F2}.

We now introduce our cascadic MG (CMG) algorithm for computing the Fielder vector.  Our CMG algorithm is a purely algebraic approach, and the multilevel structure is constructed from the graph directly.  Therefore, similar to a standard algebraic MG (AMG) method, the new algorithm consists of three steps: a setup phase, a solving phase on the coarsest level, and a cascadic solving phase (also called refinement phase in our paper). The process works as follows:

\noindent 
{\bf Step 1}:  Coarsen our graph $G_0$ iteratively to coarse graphs $G_1,G_2,...,G_J$. 

\noindent
Taking inspiration from AMG coarsening and graph matching, we introduce a technique we call heavy edge coarsening (HEC). At each level $i$, for the graph $G_i$ with $n_i$ vertices, this coarsening procedure produces aggregates $G_i^m$, $m=1,2,\cdots,n_{i+1}$ and restriction matrix $ I_i^{i+1} \in \mathbb{R}^{n_{i+1} \times n_i}$ defined by
\[
(I_i^{i+1})_{pq} = 1, \quad \text{if} \ q \in G_i^p, \quad\mbox{and}\quad  
(I_i^{i+1})_{pq} = 0, \quad \text{if} \ q \notin G_i^p.
\]

The transpose of the restriction matrix is known as a prolongation.  The coarser graph $G_{i+1}$ is defined by designating the aggregates as the vertices of the coarse graph.  Two aggregates are connected on this coarse graph if and only if there is an edge from $G_i$ connecting a vertex from one aggregate and with a vertex from the other aggregate. This creates a multilevel structure of coarse Laplacians $L^0,L^1,....,L^J$ where $L^{i+1} = I_i^{i+1}L^i (I_i^{i+1})^T$. In general, the choice of aggregates in the coarsening phase of a multilevel algorithm of this form tends to be the most expensive part of the procedure.

\noindent 
{\bf Step 2}:  Solve for the Fiedler vector on the coarse graph $G_J$. 

\noindent
{\bf Step 3}: For  $j=J$ to $j=1$ we prolongate the Fiedler vector from the coarse graph $G_j$ to the finer graph  $G_{j-1}$ and use the prolongated vector as an initial guess for a simple iterative procedure (such as power iteration). Such steps we call a ``refinement'' (or smoothing). We note that since we aim to approximate the Fiedler vector, we need to keep the iterates orthogonal to the constant vector.

\begin{Remark}
In practice the coarse graph tends to be small in size (usually $\vert V \vert < 100$). The technique implemented on this level is not extremely relevant for single computations of the vector. However, for applications which may require this eigenvalue computation a large number of times (such as recursive spectral bisection, for large $k$), this becomes more of a relevant issue. The commonly used eigensolver on this coarse level, in the absence of a good intial guess, is the Lanczos algorithm. However, in our implementation we over-coarsen to $\vert V \vert < 25$ and use power iteration on a random vector, sampled from a Gaussian distribution. 
\end{Remark}

Traditionally, in the MG method literature, \textbf{Step 2} and \textbf{3} together are called the CMG method. However, in non-spectral methods for graph partitioning, this is not the case, and for this reason we maintain the three-step structure that is prevalent in the literature.  We present the core of our cascadic eigensolver in Algorithm~\ref{alg:mfe}. 

\begin{algorithm}
\caption{Multilevel Cascadic Eigensolver}
\label{alg:mfe}
\begin{algorithmic}[1]
\STATE { \bf Input: } graph Laplacian matrix $L^0 \in \R^{n_0 \times n_0}$ \\
\STATE { \bf Output: } approximate Fiedler vector $\tilde y^{(0)}$ \\ 
\medskip
\STATE {\bf Step 1: Setup Phase} \\
\STATE set $i=0$ \\
\WHILE{$n_i>25$}
\STATE $I_i^{i+1} \gets \text{HEC}(L^i)$ \\
\STATE $L^{i+1}= I_i^{i+1} L^i  (I_i^{i+1} )^T$ \\
\STATE $i=i+1$
\ENDWHILE
\STATE $J \gets i$ \\ 
\medskip
\STATE {\bf Step 2: Coarsest Level Solving Phase} \\
\STATE $\tilde y^{(J)} \gets \text{PI}(L^J,randn(n_J))$ \\ 
\medskip
\STATE {\bf Step 3: Cascadic Refinement Phase} \\
\FOR{$j=J-1$ \TO $0$}
\STATE $\hat y^{(j)}=  (I_i^{i+1} )^T \tilde y^{(j+1)}$ \\
\STATE $\tilde y^{(j)} \gets \text{PI}$($L^j$,$\hat y^{(j)}$)
\ENDFOR
\end{algorithmic}
\end{algorithm}

Here, the subroutine HEC and PI are presented later in Algorithm \ref{alg:hec} and \ref{alg:pi}, respectively. As mentioned before, because the size of the coarsest graph is very small, and power iteration is efficient, our focus is on the first and third steps.  We will first introduce the heavy edge coarsening scheme we proposed for the setup phase, and then present our cascadic refinement scheme.

\subsection{Heavy Edge Coarsening} \label{sec:HEC}
We now consider the coarsening algorithm used for the setup phase.  The goal for this step is to coarsen a graph quickly, while also maintaining some semblance of its structure. In practice, the coarsening procedure tends to dominate the run time of the multilevel eigensolver. 
To increase the efficiency of a coarsening algorithm one needs to make compromises between fast (with respect to computational time) and optimal (with respect to better representations of the graph on coarser graphs) coarsening techniques.  We propose a new coarsening algorithm which combines ideas and algorithms described in the literature \cite{KCH,KK1,KK2} and balances between reducing computational time and providing coarse graphs with good quality. In order to introduce our coarsening algorithm, we begin by considering matching as a coarsening technique. The formal concept of a matching is as follows: 

\bdf
Let $G=(V,E)$. A matching is a subset $E^* \subset E$, such that no two elements of $E^*$ are incident on the same vertex. A matching $E^*$ is said to be a maximal matching if there does not exist an edge $e_{i,j} \in E \backslash E^*$ such that $E^* \cup \{ e_{i,j} \} $ is still a matching. 
\edf

For our purposes, the matching computed at each level is always a maximal matching. A matching is computed at each level, and the edges in the matching are collapsed to form the coarser graph. We consider the class of matching algorithms concerned with finding the matching with the heaviest edge weight.  A matching of heavy edges would make an ideal coarse graph for our multilevel eigenproblem. The reason for this is related to graph partitioning. This coarsening procedure creates a smaller edge cut on coarse levels for partitions, which results in smaller edge cuts for the finer graphs. Even though we are not refining partitions, this concept still applies, due to the close connection between the Fiedler vector and graph partitioning.  To do a matching of heavy edges optimally is rather expensive because it would require searching for the heaviest weighed edge incident to two unmatched vertices at each step. In practice, the vertices are usually visited in a random order, and the heaviest weighed incident edge with an unmatched vertex is chosen. Such a technique produces a less optimal partition, but is much faster.  We adopt a similar procedure in our coarsening algorithm.   

However, we choose to perform a more aggressive coarsening procedure, rather than matching, because it reduces the number of levels in the multilevel scheme. In addition, when considering using heavy edge schemes, an aggressive coarsening procedure (see Algorithm \ref{alg:hec}) 
is significantly easier to implement than its matching counterpart because we consider mapping each vertex to a vertex incident with it with heaviest edge, rather than picking the heaviest edge with an unmatched vertex.

We visit the vertices in a random order. At each vertex we visit, we check if it has been mapped to some aggregate. If the vertex is unmapped, we map the vertex to the aggregate containing the adjacent vertex with the heaviest connecting edge. If the vertex already belongs to an aggregate, we skip it and continue to the next vertex. We finish when all vertices have been visited and belong to some aggregate. In general, this will not result in a matching. We call this technique heavy edge coarsening (HEC) and for the specific details regarding its implementation, we refer to Algorithm \ref{alg:hec}.

\begin{Remark}
As an example, for a graph Laplacian corresponding to an anisotropic problem, one would end up with aggregates that contains vertices in lines pointing in the ``strong'' direction. This procedure would effectively only coarsen in the ``strong" direction initially.
\end{Remark}

The HEC procedure proves to be a fast and efficient means of coarsening. The structure of the finer graph is well represented, making the refinement process of power iteration converge quickly. In addition, one of the biggest benefits of HEC is the relatively small number of coarse levels required. We will introduce this concept, in the form of a lemma. 

\blm
Let $G_i=(V_i,E_i)$ be a connected graph with $n_i$ vertices. Let $n_{i+1}^{HEC}$ be the number of aggregates formed by heavy edge coarsening, and $n_{i+1}^{M}$ be the number of aggregates formed by matching. Define the coarsening rate $k^i_{HEC}= n^{HEC}_{i+1} /  n_i$ and $k^i_{M}= n^{M}_{i+1} / n_i$, respectively, Then we have $1/n_i \le k^i_{HEC} \le 0.5$ and $0.5 \le k^i_M \le 1$. 
\elm
\begin{proof} From the definition of a matching, we have that $n^{M}_{i+1}$ cannot be less than half of $n_i$. For the bounds on $k^i_{HEC}$, we note that for our HEC algorithm, every node in $V_i$ is mapped to another node, or has been mapped to, which implies that each aggregation has at least two vertices, i.e. the average $n_i/n_{i+1}^{HEC}$ is bigger than or equal to $2$.  Therefore, $n^{HEC}_{i+1}$ is at most half of $n_i$. The lower bound results from taking a HEC procedure on a graph $G_i$ such that $G_{i+1}$ is a single node. 
\end{proof}

\begin{algorithm}
\caption{Heavy Edge Coarsening (HEC)}
\label{alg:hec}
\begin{algorithmic}[1]
\STATE { \bf Input: } graph Laplacian matrix $L^i \in \R^{n_i \times n_i}$ \\
\STATE { \bf Output: } restriction matrix $I_i^{i+1}$  \\ 
\medskip
\STATE $c\gets0$\\
\STATE $p\gets randperm(n_i)$\\
\STATE $q\gets zeros(n_i,1)$\\
\FOR{$i=1$ \TO $n_i$}
\IF{$q(p(i))=0$}
\STATE $m \gets argmin(L(:,p(i)))$\\
\IF{$q(m)=0$}
\STATE $c \gets c+1$\\
\STATE $q(m)=c$\\
\STATE $q(p(i))=c$\\
\ELSE
\STATE $q(p(i))=q(m)$
\ENDIF
\ENDIF
\ENDFOR
\medskip
\STATE $I_i^{i+1} \gets zeros(c,n_i)$ \\
\FOR{$i=1$ \TO $n_i$}
\STATE $I_i^{i+1}(q(i),i)=1$
\ENDFOR
\end{algorithmic}
\end{algorithm}

We have given a bound for the value of $k_{HEC}$ (we drop the superscript $i$ for simplicity). Given below in Table \ref{tab:sample_k} are samples of what values $k_{HEC}$ takes in practice for different graphs. As expected, the values taken in practice are significantly below the given bound of $0.5$.

\begin{table}[ht]
\caption{Sample values of $k_{HEC}$} \label{tab:sample_k}
\begin{center}
\begin{tabular}{ l | c }
Graph & Sample $k^0_{HEC}$ Value \\ \hline \hline
144 & 0.1893 \\ 
598a & 0.2024 \\ 
auto & 0.1742  \\ 
\end{tabular}
\end{center}
\end{table}

What remains to be explored is the properties of the restiction matrix $I_i^{i+1}$. The most important fact that we require is that the coarse matrix created by the restriction matrix is still a Laplacian matrix of the coarse graph. In addition, we want to inspect whether or not the constant eigenvector $\bone = (1,...,1)^T$ is preserved under restrictions and prolongations. We also consider issues of orthogonal solutions with respect to the refinement procedure.  Those properties are summarized in the following proposition (see also~\cite[Theorem~3.6]{2003KimH_XuJ_ZikatanovL-aa} for such results). 
 
\begin{Proposition}
\label{prop}
Let $I_i^{i+1} \in \R^{n_{i+1} \times n_i}$ be a restriction matrix defined by HEC. Then we have the following:
\begin{enumerate}
\item $( I_i^{i+1})^T \bone^{i+1} = \bone^{i}$. That is, the eigenvector $\bone$ is preserved under refinement. 
\item If $L^{i}$ is a Laplacian matrix, then $L^{i+1} =  I_i^{i+1} L^i  (I_i^{i+1} )^T$ is also a Lapacian matrix. In particular, $L^{i+1} \bone^{i+1} = 0$.
\item Let $u \in \bone^{\perp} = \{ u \vert (u,\bone)=0 \} \subset \R^{n_i}$. Then  $I_i^{i+1} u \in \bone^{\perp} \subset \R^{n_{i+1}}$. However, in general, $(I_{i-1}^{i})^T u \notin \bone^{\perp} \subset \R^{n_{i-1}}$. 
\end{enumerate}
\end{Proposition}

\begin{proof} We begin with (1). This follows from the fact that each vertex in $V_i$ is mapped to only one vertex in $V_{i+1}$.  However, $ I_i^{i+1} \bone^{i} \ne \bone^{i+1}$ . This is expected, as the number of vertices in $V_i$ mapped to a given vertex $v_j \in V_{i+1}$ varies. 

To prove (2), we need to show that $L^{i+1}$ is still symmetric, with positive diagonal and non-positive offdiagonal, with $L^{i+1} \bone^{i+1} = 0$. We begin by decomposing $L^i$ into its degree matrix $D^i$ and adjacency matrix $A^{i}$. This gives us $L^{i+1} =  I_i^{i+1} D^i  (I_i^{i+1} )^T -  I_i^{i+1} A^i  (I_i^{i+1} )^T$. $ I_i^{i+1} D^i  (I_i^{i+1} )^T $ is still a degree matrix, and $I_i^{i+1} A^i  (I_i^{i+1} )^T$ an adjacency matrix. We show that $L^{i+1}$ is a Laplacian by taking $L^{i+1} \bone^{i+1} = I_i^{i+1} L^i  (I_i^{i+1} )^T \bone^{i+1} = I_i^{i+1} L^i \bone^i = 0$. 

Part (3) of the Proposition can be shown as follows. Let $u \in \bone^{\perp} \subset \R^{n_i}$. We have $(I_i^{i+1} u, \bone^{i+1})=(u,(I_i^{i+1})^T \bone^{i+1})=(u,\bone^i)=0$. Therefore, $I_i^{i+1} u \in \bone^{\perp} \subset \R^{n_{i+1}}$. Looking at $(I_{i-1}^{i})^T u$, we see $ ( (I_{i-1}^{i})^T u, \bone^{i-1} ) = (u, I_{i-1}^{i} \bone^{i-1} ) \ne (u, \bone^{i} ) = 0$, since  $I_{i-1}^{i} \bone^{i-1} \ne \bone^{i}$. 
\end{proof}

\subsection{Refinement (Smoothing) Strategies} \label{s:pi}
Given an approximate Fiedler vector $y^{(i+1)}$ on a coarse graph $G_{i+1}$, we aim to find an optimal manner to project this vector back to the finer graph $G_i$ and refine it to an approximate Fiedler vector $y^{(i)}$ on $G_i$. We begin by considering the projection problem. The most natural way to project $y^{(i+1)}$ to $G_i$ is to use the restriction matrix $I_i^{i+1}$ obtained from coarsening, define our prolongation matrix to be $(I_i^{i+1})^T$, and let the initial approximation be $\tilde y^{(i)} = (I_i^{i+1})^T y^{(i+1)}$. However, we have to concern ourselves with orthogonality to the eigenvector $\bone$. From Proposition \ref{prop}, we have that  $( \tilde y^{(i)}, \bone^{i} ) \ne 0$. Therefore, before we can perform any sort of eigenvalue refinement procedure, we require our inital vector to be in the subspace $\bone^{\perp} =   \{ u \vert (u,\bone)=0 \}$. This can be accomplished by one iteration of Gram-Schmidt. From here, the orthogonality will be approximately maintained, since $\bone ^{\perp}$ is $L$-invariant. 

Given an approximation $\tilde y^{(i)}$, we can refine it in a number of ways. We consider power iteration as a refinement scheme in our CMG algorithm because of its simplicity. In this way we take advantage of the sparsity of our Laplacian.

Because the Fiedler vector corresponds to the second smallest eigenvalue of the graph Laplacian, we cannot apply the power iteration directly.  Therefore, we compute a Gershgorin bound on the eigenvalues of a Laplacian $L$ by considering $g=\|L\|_{\ell^1}$. From the Gershgorin circle Theorem and properties of the Laplacian, we have that all the eigenvalues of $gI-L$ are positive, with eigenvalues $g- \lambda_1, g- \lambda_2,...,g- \lambda_n$. The eigenvectors obviously remain unchanged. In this way it suffices to perform power iteration on $gI-L$, coupled with an intial orthogonalization to $\bone$. We note that $\bone ^{\perp}$ is also invariant under $gI-L$. This variant of power iteration is detailed in Algorithm \ref{alg:pi}. 

We proceed by examining the convergence for power iteration. Let $u^0$ denote our initial guess, and $u^k$ represent the normalized vector resulting from $k$ iterations. For our algorithm, the stopping criterion is given by $(u^k,u^{k-1}) >1-\delta$, for some given tolerance $\delta$. We note that this is equivalent to $||u^k - u^{k-1}||^2<2\delta$. We recall the following result, with respect to power iteration on an arbitrary symmetric matrix.

\bthm \label{powerit}
Let $A$ be a symmetric matrix with eigenvalues $\lambda_1 > \lambda_2 \ge ... \ge \lambda_n \ge0$ and corresponding eigenvectors $\varphi_1,\varphi_2,...,\varphi_n$. Then power iteration, with intial guess $u^0$, $(u^0,\varphi_1) \ne 0$, has convergence rate given by
\begin{equation*}
\sin \angle (u^k,\varphi_1)< \big| \frac{\lambda_2} {\lambda_1} \big|^k \tan \angle (u^0,\varphi_1),
\end{equation*}
where $\angle (u, v)$ is the angle between the subspaces spanned by $u$ and $v$.
\ethm

A proof of this result can be found in \cite{GV}. We see that the number of iterations required depends on the eigenvalue gap, the quality of the initial guess in our multilevel structure, as well as the chosen tolerance. For general graphs it is hard to obtain better estimates for the power iteration portion of the cascadic algorithm. This stems mainly from the fact that the eigenvalues of a general graph does not follow any set spacing or structure, and that our aggregation procedure is random in nature, making an estimate of the quality of the initial approximation extremely tough in practice. This limits our ability to give rigorous theoretical results for our algorithm in general. In Section~\ref{sec:GCMG} we will give results for our cascadic eigenvalue algorithm for the case of graphs resulting from elliptic PDE discretizations, with geometric coarsening as the cascadic coarsening procedure and a fixed number of power iteration steps at each level. These simplifications remove the barriers that we currently face for analysis. However, we will give numerical justification that these results are robust to general graphs with HEC as the coarsening procedure.

\begin{algorithm}
\caption{Power Iteration (PI)}
\label{alg:pi}
\begin{algorithmic}[1]
\STATE { \bf Input: } graph Laplacian matrix $L \in \R^{n \times n}$, initial guess $\tilde y^0$ \\
\STATE { \bf Output: } approximate Fiedler vector $\tilde y$ \\ 
\medskip
\STATE $g= max_{i} \sum_{1 \le j \le n} \vert l_{i,j} \vert$\\
\STATE $B_g=gI-L$\\
\STATE $u=\tilde y^0 - \frac {\bone^T \tilde y^0} {n} \tilde y^0$\\
\STATE $u$ $\leftarrow$ $\frac{u} {\| u \|}$\\
\STATE $v$ $\leftarrow$ $zeros(n,1)$\\
\WHILE{$u^T v < 1 - tol$}
\STATE $v$ $\leftarrow$ $u$\\
\STATE $u=B_g v$\\
\STATE $u$ $\leftarrow$ $\frac{u} {\| u \|}$
\ENDWHILE
\STATE $\tilde y = u$ \\
\end{algorithmic}
\end{algorithm}

\section{Convergence Analysis of CMG for Elliptic Eigenvalue Problems} \label{sec:GCMG}
In Section \ref{sec:CMG-Fiedler}, we introduced the CMG method for computing the Fiedler vector of a graph Laplacian.  However, we used a purely algebraic coarsening strategy (see Section \ref{sec:HEC}) to construct the hierarchical structure; hence, similar to the AMG method for the Poisson problem, the convergence analysis for a purely algebraic CMG method is difficult. In order to illustrate and theoretically justify the convergence of the proposed CMG method, we discuss the geometric CMG (GCMG) method for the elliptic eigenvalue problem.  
As a model which shares a great deal of properties with the graph Laplacian eigenproblem, we consider the following elliptic eigenvalue problem with Neumann boundary conditions, 
\begin{equation} \label{lapl}
        - \Delta \varphi = \lambda \varphi, \qquad \text{on} \ \Omega, \qquad
\frac { \partial \varphi} { \partial n} = 0, \qquad \text{on} \  \partial \Omega 
  \end{equation}
where $\Omega \in \R^d$ is a polygonal Lipschitz domain.  We only consider the two- and three- dimensional case to illustrate the theoretical bounds that can be obtained for the cascadic multilevel algorithm. However, the GCMG method we discussed here can be naturally applied for higher dimentional cases.   Using the standard Sobolev space $H^1(\Omega)$, we consider the weak formulation of \eqref{lapl} as follows: find $(\lambda, \varphi) \in \R \times H^1(\Omega)$ such that
\begin{equation}\label{eqn:Neumann-weak}
a(\varphi,v) = \lambda(\varphi,v), \quad \forall \ v \in H^1(\Omega),
\end{equation}
where the bilinear form $a(u, v) = (\nabla u, \nabla v)$, and $(\cdot, \cdot)$ is the standard $L^2$ inner product.  Here, the bounded symmetric bilinear form $a(\cdot,\cdot)$ is coercive on the quotient space $H^1(\Omega)$, and, therefore, induces an energy-norm as follows:
\begin{equation}\label{def:a-norm}
\| u \|_a^2 = a(u,u), \quad \forall \ u \in H^1(\Omega)\backslash \R.
\end{equation}
Moreover, we denote the $L^2$-norm by $\| \cdot \|$ as usual.  Similar to the eigenvalues for the graph Laplacian, $\lambda = 0$ is also an eigenvalue of the eigenvalue problem \eqref{eqn:Neumann-weak}, We can order the eigenvalues as follows: $0= \lambda^{(1)} \le \lambda^{(2)} \le ...  $ and denote by $\varphi^{(1)}, \varphi^{(2)}, ... $ the corresponding eigenfunctions.  Again, we are interested in approximating the second smallest eigenvalue of \eqref{eqn:Neumann-weak} and its corresponding eigenfunction space. 

Given a nested family of quasi-uniform triangulations $\{  \Gamma_j \}_{j=0}^J$, namely, $$\frac{1}{c}2^{j-J} \leq h_j = \max_{T \in \Gamma_j} \text{diam}(T) \leq c 2^{j-J},$$
the spaces of linear finite elements are
\begin{equation*}
V_j = \{u \in C(\Omega): u|_T \in P_1(T), \forall \ T \in \Gamma_j \},
\end{equation*}
where $P_1(T)$ denotes the linear functions on the triangle $T$.  We have
\begin{equation*}
V_J \subset V_{J-1} \subset \cdots \subset V_0 \subset H^1(\Omega).
\end{equation*}
The finite element approximations of \eqref{eqn:Neumann-weak} on each level are as follows: find $(\lambda_j, \varphi_j) \in \R \times V_j$ such that 
\begin{equation}\label{eqn:fem-eigen}
a(\varphi_j, v_j) = \lambda_j (\varphi_j, v_j), \quad \forall \ v_j \in V_j.
\end{equation}
We can order the eigenvalues as follows: $0= \lambda_j^{(1)} \le \lambda_j^{(2)} \le \cdots \le \lambda_j^{(N_j)} $ and denote by $\varphi^{(1)}, \varphi^{(2)}, ... \varphi^{(N_j)} $ the corresponding eigenfunctions.   Again, we are interested in approximating the second smallest eigenpair on the finest level.  Moreover, we can define an operator $A_j$ by $a(u_j,v_j) = (A_j u_j, v_j)$, $\forall u_j, v_j \in V_j$. 

We assume the elliptic eigenvalue problem has $H^{1+\alpha}$-regularity, i.e., the eigenvalue function $\varphi \in H^{1+\alpha}$ for some $0 < \alpha \leq 1$.  Then we have the following error estimates regarding the standard finite element approximation of the elliptic eigenvalue problem, taken from the work of Babu\v{s}ka and Osborn~\cite{BO}.

\blm \label{ellip}
Assume that $(\lambda_h, \varphi_h) \in (\R \times V_h)$ is a finite element approximation of \eqref{eqn:Neumann-weak}. Then we have 
\begin{enumerate}[(i)]
\item $ | \lambda -\lambda_h | \leq C h^{2 \alpha}$,
\item there exists an eigenfunction $\varphi$ corresponding to $\lambda$, such that
\begin{equation}\label{ine:error-eigenfuction}
\| \varphi - \varphi_h \|_a \leq C h^{\alpha}, 
\end{equation}
\end{enumerate}
where $C$ is a constant that does not depend on the mesh size.
\elm

Now we introduce the Ritz projection on level $j$ by $a(P_{j} u, v_{j}) = a(u, v_{j})$, $\forall v_{j} \in V_{j}$.  We assume the eigenvalue $\lambda^{(l)}$ we want to approximate has multiplicity $k$, i.e. $\lambda^{(l)} = \lambda^{(l+1)} = \cdots = \lambda^{(l+k-1)}$ and there are $k$ corresponding eigenfunctions $\varphi^{(l)}, \varphi^{(l+1)}, \varphi^{(l+k-1)}$,  then on level $j$, there are $k$ approximate eigenpairs $(\lambda_j^{(l+i)}, \varphi_j^{(l+i)})$, $i=0, \cdots k-1$, such that $\lambda_j^{(l)} \le \lambda_j^{(l+1)} \le \cdots \le \lambda_j^{(l+k-1)}$.  Let $Q_j$ denote the $L^2$-projection onto $\text{span} \{  \varphi_j^{(l)},  \varphi_j^{(l+1)}, \cdots,  \varphi_j^{(l+k-1)} \}$ and define $\Lambda_j := Q_j \circ P_j$.  Then for an eigenfunction $\varphi^{(l+i)}$, $i=0, 1, \cdots, k-1$, $\Lambda_j \varphi^{(l+i)} \in \text{span} \{  \varphi_j^{(l)},  \varphi_j^{(l+1)}, \cdots,  \varphi_j^{(l+k-1)} \}$ is regarded as its approximation.  The following best-approximation result of $\Lambda_j \varphi^{(l+i)}$ can be found in \cite{Gallistl.D.2014a}.  For the simplicity of the presentation, we omit the superscript $(l+i)$.

\blm \label{lem:ba}
Assume that $h_j$ is sufficiently small and the elliptic eigenvalue problem has $H^{1+\alpha}$-regularity, then for any eigenpair $(\lambda, \varphi)$ with $\| \varphi \|= 1 $, we have
\begin{equation}\label{ine:ba}
\| \varphi - \Lambda_j \varphi \|_a \leq C h_j^{\alpha},
\end{equation}
where $C$ is a constant that does not depend on the mesh size.
\elm

Next, we will present several results related to the approximation property of finite element approximate eigenfunctions between two successive levels $j$ and $j+1$.  For the sake of simplicity, we will use script $h$ to denote level $j$ and script $H$ to denote level $j+1$.  Moreover, we denote the mesh size $h_j$ by $h$ and $h_{j+1}$ by $H$.  Considering the eigenvalue problem on level $j+1$ as a finite element approximation of the eigenvalue problem on level $j$, we have the following lemma regarding the approximation in the energy norm.

\blm \label{lem:discrete-ba} Let $\{ (\lambda_h^{(l+i)}, \varphi_h^{(l+i)})\}_{i=0}^{i=k-1}$ and $\{ (\lambda_H^{(l+i)}, \varphi_H^{(l+i)}) \}_{i=0}^{i=k-1}$ be approximate eigenpairs of the eigenvalue $\lambda^{(l)}$  with multiplicity $k$.  For sufficiently small $H$ and   for any $w_h \in \text{span}\{ \varphi_h^{(l)}, \varphi_h^{(l+1)}, \cdots, \varphi_h^{(l+k-1)}\}$ 
we have 
\begin{equation}\label{ine:discrete-ba} \| w_h - \Lambda_Hw_h \|_a \leq C H^{\alpha}, 
\end{equation} 
\elm 
\begin{proof} 
Setting $w_h = \sum_{i=0}^{k-1} \beta_i \varphi_h^{(l+i)}$ we have, 
\begin{equation*} 
\| w_h - \Lambda_H w_h \|_a = \| \sum_{i=0}^{k-1} \beta_i \left( \varphi_h^{(l+i)} - \Lambda_H \varphi_h^{(l+i)} \right) \|_a \leq \sum_{i=0}^{k-1} |\beta_i| \| \left( \varphi_h^{(l+i)} - \Lambda_H \varphi_h^{(l+i)} \right) \|_a \leq C H^{\alpha}.  
\end{equation*} 
This completes the proof.  
\end{proof}

The next lemma provides an estimate on the error of approximation 
$(w_h-\Lambda_H w_h)$ in the $L^2$ norm.  In the proof, we use a separation bound 
given in Boffi~\cite{BOF}, namely, that for sufficiently small $H$ the following estimate holds:
\begin{equation}\label{def:seperation} 
\frac{ | \lambda_h^{(l)} |}{ | \lambda_h^{(l)} - \lambda_{H}^{(i)} |} \le d_l < \infty, \quad \text{for all } i \ne l,l+1,...,l+k-1. 
 \end{equation}

The $L^2$ estimate is then as follows. 
\blm \label{lem:l2-error}
Let $\{ (\lambda_h^{(l+i)}, \varphi_h^{(l+i)})\}_{i=0}^{i=k-1}$ and $\{ (\lambda_H^{(l+i)}, \varphi_H^{(l+i)}) \}_{i=0}^{i=k-1}$ be approximate eigenpairs of the eigenvalue $\lambda^{(l)}$ with multiplicity $k$.  For sufficiently small $H$ and for any $w_h \in \text{span}\{ \varphi_h^{(l)}, \varphi_h^{(l+1)}, \cdots, \varphi_h^{(l+k-1)}\}$, we have
\begin{equation}\label{ine:l2-error}
\|(I-\Lambda_H) w_h \| \leq C \|(I - P_H) w_h \|,
\end{equation}
where $C$ is a constant that does not depend on the mesh size.
\elm

\begin{proof}
  Let us first set $w_h = \sum_{j=l}^{l+k-1} \beta_j \varphi_h^{(j)}$.  Because $P_H w_h \in V_H$, we have $P_H w_h = \sum_{i=1}^{N_H} \alpha_i \varphi_H^{(i)}$ where $\alpha_i = (P_Hw_h, \varphi_H^{(i)})$.  Since by the definition of $Q_H$ we have that 
$Q_H\varphi_H^{(l+i)}=\varphi_H^{(l+i)}$ for $i=0,\ldots,k-1$, it is straightforward to calculate that 
\begin{equation*}
P_H w_h - \Lambda_H w_h = \sum_{i \neq l, l+1, \cdots, l+k-1} \alpha_i \varphi_H^{(i)}.
\end{equation*}
Next, using the relation
  \begin{equation*}
\lambda_H^{(i)} (P_H \varphi_h^{(j)}, \varphi_H^{(i)}) = a(P_H \varphi_h^{(j)}, \varphi_H^{(i)}) = a(\varphi_h^{(j)}, \varphi_H^{(i)}) = \lambda_h^{(j)} (\varphi_h^{(j)}, \varphi_H^{(i)}),
\end{equation*}
we obtain that
\begin{equation*}
(\lambda_H^{(i)} - \lambda_h^{(j)}) (P_H \varphi_h^{(j)}, \varphi_H^{(i)}) = \lambda_h^{(j)} (\varphi_h^{(j)} - P_H \varphi_h^{(j)}, \varphi_H^{(i)}).
\end{equation*}
Therefore,
\begin{align*}
\| P_H w_h - \Lambda_H w_h\|^2 & = (P_H w_h, P_Hw_h - \Lambda_H w_h) = \sum_{i\neq l, l+1, \cdots, l+k-1} (P_H w_h, \varphi_H^{(i)})^2 \\
& = \sum_{i\neq l, l+1, \cdots, l+k-1} \left( \sum_{j=l}^{l+k-1} \beta_j (P_H \varphi_h^{(j)}, \varphi_H^{(i)}) \right)^2 \\
&= \sum_{i\neq l, l+1, \cdots, l+k-1} \left( \sum_{j=l}^{l+k-1} \beta_j \frac{\lambda_h^{(j)}}{\lambda_H^{(i)} - \lambda_h^{(j)}} ( \varphi_h^{(j)} - P_H \varphi_h^{(j)}, \varphi_H^{(i)}) \right)^2 \\
&\leq d_l ^2 \left(   \sum_{i\neq l, l+1, \cdots, l+k-1}  (w_h - P_H w_h,  \varphi_H^{(i)})^2   \right) \\
& = d_j^2 \| w_h - P_H w_h \|^2.
\end{align*}
And we have
\begin{equation*}
\| w_h - \Lambda_H w_h \| \leq \| w_h - P_H w_h\| + \| P_H w_h - \Lambda_H w_h \| \leq (1 + d_l) \| (I - P_H)w_h \|,
\end{equation*}
which leads to \eqref{ine:l2-error} with $C = 1 + d_l$.
\end{proof}

Based on Lemma \ref{lem:l2-error} and the interpolation argument \cite{Berg.J;Lofstrom.J.1976a}, we have the following \emph{approximation property} for the eigenvalue problem.
\blm \label{approxprop}
Let $\{ (\lambda_h^{(l+i)}, \varphi_h^{(l+i)})\}_{i=0}^{i=k-1}$ and $\{ (\lambda_H^{(l+i)}, \varphi_H^{(l+i)}) \}_{i=0}^{i=k-1}$ be approximate eigenpairs of the eigenvalue $\lambda^{(l)}$ with multiplicity $k$.  Assuming that $H$ is sufficiently small,  for any $w_h \in \text{span}\{ \varphi_h^{(l)}, \varphi_h^{(l+1)}, \cdots, \varphi_h^{(l+k-1)}\}$, we have
\begin{equation}
\|(I- \Lambda_H) w_h \|_{H^{1-\alpha}} \leq C H^{\alpha} \| (I - \Lambda_H) w_h \|_a
\end{equation}
where $C$ is a constant independent of the mesh size.
\elm

\begin{proof}
From Lemma \ref{lem:l2-error}, we have
\begin{align*}
\| (I-  \Lambda_H) w_{h} \| &\le  
C \| (I - P_{H} ) w_h \|  = 
C \| (I-P_H)[(I - P_{H} ) w_h] \| \\
  &\le  C H \| (I - P_{H} ) w_h\|_{a} \le  C H \| (I -  \Lambda_H)w_{h} \|_{a},
\end{align*}
where the last inequality follows from noting that $ \| (I - P_{H} ) w_h \|_{a} = \inf_{v \in V_{H} } \|w_h - v \|_a$. By an interpolation argument, the desired result follows.
\end{proof}

Based on the nested spaces $V_J \subset V_{J-1} \subset  \cdots \subset V_0$, the GCMG method for eigenvalue problems seeks to solve the eigenvalue problem exactly on the coarse grid $V_J$, and interpolate and smooth the approximation back to the fine grid $V_0$.  In this section, we consider the GCMG method, and therefore, the geometric prolongation and restriction are used in our algorithm, and will be omitted as usual. Our cascadic Algorithm \ref{alg:mfe} can be framed as follows: 

\begin{algorithm}
\caption{Geometric Cascadic Multigrid Method for Elliptic Eigenvalue Problem}
\label{alg:casc}
\begin{algorithmic}[1]
\IF{$j=J$ (coarsest level)}
\STATE solve $a(\varphi_J, v_J) = \lambda (\varphi_J, v_J)$ exactly, and let $u_J := \varphi_J^{(l)}$
\ELSE
\STATE $u_{j}=( I - \omega_j A_{j})^{k_j}  u_{j+1}$, where $\omega_j = \| A_j \|^{-1}_{\infty}$.    (with appropriate scaling)  \\
\STATE $\lambda_j = \frac{a(u_j, u_j)}{(u_j, u_j)}$
\ENDIF
\end{algorithmic}
\end{algorithm}

\begin{Remark}
We present the algorithm for just computing one approximate eigenpair.  However, we can easily extend the algorithm to compute several approximate eigenpairs by starting with $k$ approximate eigenpairs on the coarest level and then, on each level, after smoothing each approximate eigenfunction, we can orthogonalize them and compute corresponding Rayleigh quotients.  
\end{Remark}

This procedure is only performed once and results in the approximation $u_0 \in V_0$. Next, we consider the uniform convergence of the proposed GCMG method (Algorithm \ref{alg:casc}).  Our analysis will follow the standard convergence analysis for the CMG method for elliptic partial differential equations. We will first present a two-level error estimate on two successive levels $j+1$ and $j$, and then generalize it to the multilevel case later.  Again, we use $h$ to denote $j+1$ and $H$ to denote $j$ for the sake of simplicity.  We begin by recalling the following lemma.
\blm
\label{calc}
For any $k \in \Z_+$, we have $\max_{t \in [0,1]} t(1-t)^k < \frac{1} {k+1}$. 
\elm

This is a simple result, and is used often in multigrid literature. Denoting by 
$S_h = I - \omega_h A_h$ the error propagation operator associated with the Richardson smoother, we have the following \emph{smoothing property}.

\blm
\label{richresult}
Let $\omega = ||A_h||^{-1}_\infty$ and $k$ be the number of smoothing steps. Then the following estimate holds
\begin{equation} \label{ine:smoothing-prop}
\| S_h^{k} v_h \|_a  \leq C  \frac{h^{-\alpha}}{k^{\alpha/2}}  \| v_h \|_{H^{1-\alpha}},   \quad \forall \ v_h \in V_h \backslash \R.
\end{equation}
\elm
\begin{proof}  
Recall that, by the properties of a graph Laplacian,  $A_h$ is Hermitian 
and positive semi-definite, and, moreover,  $A_h$ is positive definite 
on the subspace $\{ u\; | \; (u, \bone) =0 \}$. Hence, 
\begin{eqnarray*}
|| S_h^\nu u ||^2_a &=& \big( (I - \omega A_h)^\nu u , (I - \omega A_h)^\nu u \big)_a \\
&=& \big(  A_h (I - \omega A_h)^\nu u , (I - \omega A_h)^\nu u \big) \\
&=& \omega^{-1} \big(\omega A_h (I - \omega A_h)^{2 \nu} u , u \big).
\end{eqnarray*}
Noting that the spectral radius $\rho (\omega A_h) \le 1 $, $\omega^{-1} \eqsim h^{-2}$ and making use of Lemma \ref{calc}, we obtain
$$ || S_h^\nu u ||^2_a \lesssim h^{-2} \eta_0 (2 \nu) ||u||^2. $$
This gives us
$$
\| S_h^{k} v_h \|_a \leq C \frac{h^{-1}}{k^{1/2}} \| v_h \|,   
\quad \forall \ v_h \in V_h \backslash \R. 
$$
Recalling that $S_h$ is a contraction, and, hence, 
$\| S_h^{k} v_h \|_a \leq C \| v_h \|_a$, for all $v_h \in V_h$, the desired result follows by an interpolation argument. 
\end{proof}

We are now able to show the uniform convergence of our GCMG Algorithm \ref{alg:casc} under suitable conditions.

\blm \label{lem:two-level-convergence}
Let $\{ (\lambda_h^{(l+i)}, \varphi_h^{(l+i)})\}_{i=0}^{i=k-1}$ and $\{ (\lambda_H^{(l+i)}, \varphi_H^{(l+i)}) \}_{i=0}^{i=k-1}$ be approximate eigenpairs of the eigenvalue $\lambda^{(l)}$ with multiplicity $k$ and $u_h$ be computed by Algorithm \ref{alg:casc}.  Assuming that $H$ is sufficiently small,  there exist $\varphi^h \in \text{span}\{ \varphi_h^{(l)}, \varphi_h^{(l+1)}, \cdots, \varphi_h^{(l+k-1)}\}$ and $\varphi^H \in  \text{span} \{  \varphi_H^{(l)}, \varphi_H^{(l+1)}, \cdots, \varphi_H^{(l+k-1)} \}$ such that the error of the two-level GCMG Algorithm \ref{alg:casc} with the Richardson smoother for the eigenvector can be estimated by
\begin{equation} \label{ine:two-level-convergence}
\|  u_h - \varphi^h \|_a \leq C \frac{h^\alpha}{k^{\alpha/2}}  + \|  u_{H} - \varphi^H \|_a,
\end{equation}
where $k$ is the number of smoothing steps and $C$ is a constant that does not depends on mesh size.
\elm

\begin{proof}
Denote $e_H = u_H - \varphi^H$, we have $u_H = \varphi^H + e_H$.  Let $\bar{\varphi}^h \in \text{span}\{ \varphi_h^{(l)}, \varphi_h^{(l+1)}, \cdots, \varphi_h^{(l+k-1)}\}$ satisfy $\varphi^H = \Lambda_H \bar{\varphi}^h $, then we have
\begin{equation*}
u_H = \bar{\varphi}^h + (\varphi^H - \bar{\varphi}^h) + e_H.
\end{equation*}
Let $\bar{\varphi}^h = \sum_{i=l}^{l+k-1} \beta_i \varphi_{h}^{(i)}$. We have
\begin{equation*}
u_h  =  S_h^k \bar{\varphi}^h + S_h^k(\varphi^H - \bar{\varphi}^h) + S_h^k e_H = \sum_{i=l}^{l+k-1} \beta_i  \left( \frac{\omega_h^{-1} - \lambda_h^{(i)}}{\omega_h^{-1}} \right)^k \varphi_h^{(i)} + S_h^k(\varphi^H - \bar{\varphi}^h) + S_h^k e_H.
\end{equation*}
Denote $\varphi^h :=  \sum_{i=l}^{l+k-1} \beta_i  \left( \frac{\omega_h^{-1} - \lambda_h^{(i)}}{\omega_h^{-1}} \right)^k \varphi_h^{(i)} \in  \text{span}\{ \varphi_h^{(l)}, \varphi_h^{(l+1)}, \cdots, \varphi_h^{(l+k-1)}\}$, we have
\begin{equation*}
e_h := u_h - \varphi^h = S_h^k(\varphi^H - \bar{\varphi}^h) + S_h^k e_H.
\end{equation*}
Therefore, 
\begin{eqnarray*}
\| e_h \|_a & \leq &\| S_h^k(\varphi^H - \bar{\varphi}^h) \|_a + \| S_h^k e_H \|_a \\
& \leq & C \frac{h^{-\alpha}}{k^{\alpha/2}}  \| \varphi^H - \bar{\varphi}^h \|_{H^{1-\alpha}} + \|e_H\|_{a} \qquad \text{(from Lemma~\ref{richresult})} \\
& \leq &C \frac{1}{k^{\alpha/2}} \| \Lambda_H \bar{\varphi}^h - \bar{\varphi}^h \|_a + \|e_H\|_a \qquad
\text{(from Lemma~\ref{approxprop})}					
\\
					& \leq &
C \frac{H^{\alpha}}{k^{\alpha/2}} + \| e_H \|_a\qquad  \text{(from Lemma~\ref{lem:ba}).}
\end{eqnarray*}
Finally, \eqref{ine:two-level-convergence} follows by noting that $2h / c \leq H \leq c 2h$.
\end{proof}

By recursively applying the two-level result Lemma \ref{lem:two-level-convergence} on two successive levels $j+1$ and $j$, we can derive the error estimate of the multilevel GCMG.  From now on, we use the script $j$ again to denote the index of the level.  Because $ 2^j h_0 / C \leq h_j \leq C 2^j h_0 $, we consider $k_j = \beta^j k_0$ for some fixed $\beta  > 0$. We have the following error estimate.

\bthm
\label{maintheor}
Let $\{  \lambda_0^{(l+i)} \varphi_0^{(l+i)}  \}_{i=0}^{k-1}$ be approximate eigenpairs of the eigenvalue $\lambda^{(l)}$ with multiplicity $k$ and $u_0$ be computed by Algorithm \ref{alg:casc}.  Let the number of smoothing steps on level $j$ be given by $k_j = \beta^j k_0$. If $h_J$ is sufficiently small, then there exists $\varphi^0 \in \text{span}\{  \varphi_0^{(l)},  \varphi_0^{(l+1)}, \cdots, \varphi_0^{(l=k-1)}\}$ such that the error of the GCMG method for the eigenvector can be estimated by
\begin{equation*}
\|  u_0  - \varphi^0 \|_a \leq 
\begin{cases}
C \frac{1}{ 1 - (4/\beta)^{\alpha/2}} \frac{h^{\alpha}_0}{k_0^{\alpha/2}}, &\quad \text{if} \ \beta > 4,  \\
C J \frac{h^{\alpha}_0}{k_0^{\alpha/2}}, &\quad \text{if} \ \beta = 4.
\end{cases}
\end{equation*}
and for the eigenvalue, by
\begin{equation*}
|  \lambda_0  - \lambda^0 | \leq 
\begin{cases}
C (\frac{1}{ 1 - (4/\beta)^{\alpha/2}})^2 \frac{h^{2\alpha}_0}{k_0^{\alpha}}, &\quad \text{if} \ \beta > 4,  \\
C J^2 \frac{h^{2\alpha}_0}{k_0^{\alpha}}, &\quad \text{if} \ \beta = 4,
\end{cases}
\end{equation*}
where $C$ denotes a constant that does not depend on the mesh size.
\ethm

\begin{proof}
Using the two level result from Lemma \ref{lem:two-level-convergence}, 
\begin{equation*}
\|  u_{j+1} - \varphi^{j+1} \|_a \leq C \frac{h_j^\alpha}{k_j^{\alpha/2}}  + \|  u_{j} - \varphi^j \|_a,
\end{equation*}
summing from $j = J-1$ to $0$, and noting that $e_J = 0$, we have
\begin{equation*}
\| u_0 - \varphi^0 \|_a  \leq  C \sum_{j = 0}^{J-1} \frac{h_j^{\alpha}}{k_j^{\alpha/2}}.
\end{equation*}
Moreover, using the identity 
\begin{equation*}
\lambda_0 - \lambda^0 = \frac{a(u_0 - \varphi^0, u_0- \varphi^0)}{(u_0, u_0)} - \lambda^0\frac{(u_0 -  \varphi^0, u_0- \varphi^0)}{(u_0, u_0)},
\end{equation*}
we have
\begin{equation*}
| \lambda_0 - \lambda^0 | \leq C  (\sum_{j = 0}^{J-1} \frac{h_j^{\alpha}}{k_j^{\alpha/2}})^2
\end{equation*}
The estimates follow directly from the following estimation
\begin{equation*}
\sum_{j=0}^{J-1} \frac{h_j^{\alpha}}{k_j^{\alpha/2}} \leq C \frac{h_0^{\alpha}}{k_0^{\alpha/2}} \sum_{j=0}^{J-1} (\frac{4}{\beta})^{\frac{j \alpha}{2}}
\end{equation*}
\end{proof}

What remains to be considered is the computational complexity.  Assuming still that $k_j = \beta^j k_0$ for some fixed $\beta  > 0$, we have the following corollary.

\bcor
Let the number of smoothing steps on level $j$ be given by $k_j = \beta^j k_0$, then the computational cost of the GCMG method is proportional to
\begin{equation*}
\sum_{j=1}^J k_j n_j  \le 
\begin{cases}
C  \frac{1}{ 1 - \beta/2^d} k_0 n_0 , &\quad \text{if} \ \beta < 2^d,  \\
C J k_0 n_0, &\quad \text{if} \ \beta = 2^d,
\end{cases}
\end{equation*}
where $d$ denotes the dimension and $C$ denotes a constant that does not depend on the mesh size.
\ecor

\begin{proof}
The result follows naturally from noting that $2^{dj}/c \le n_j \le c 2^{dj}$ and observing that
\begin{equation*}
\sum_{j=1}^J k_j n_j \le c k_0 n_0 \sum_{j=0} ^{J-1} \big( \frac{ \beta} {2^d} \big)^j
\end{equation*}
\end{proof}

We see that if we set $\beta$ to be $4<\beta<2^d$, our results regarding accuracy and complexity do not contradict. Therefore, we see that for $d=3$ our algorithm is optimal, and is sub-optimal for $d=2$.

\section{Numerical Results}
\label{numerics}

We now perform numerical tests on a variety of different graphs (listed in Table \ref{tab:test}), taken from the University of Florida Sparse Matrix Collection \cite{DH}. All of our computations were performed on a MacBook Pro PC with a 2.9 GHz Intel Core i7 Processor with 8 GB RAM. All the algorithms are implemented in the FiedComp package\footnote{\texttt{http://www.personal.psu.edu/jcu5018}}, written in MATLAB. We consider the performance of our eigensolver against the Locally Optimal Preconditioned Conjugate Gradient Method (LOPCG), with Lean Algebraic Multigrid (LAMG) as a preconditioner. The LOPCG Method is part of the MATLAB BLOPEX Package by Knyazev, and is described in \cite[Algorithm~5.1]{KN}. The LAMG preconditioner is a MATLAB package by Livine, and is described in Livine and Brandt's paper \cite{LB}. We use a residual tolerance of $.05$ for the LOPCG, and an $.1$ tolerance for the LAMG preconditioner. For our Cascadic Eigensolver, we use the tolerance $(u^k,u^{k-1}) >1-10^{-8}$. In Table \ref{tab:test} we report the run times in seconds for each graph, along with a measure of the error in the approximate eigenvector, given by $\| (L - \tilde r I) \tilde y \| $, where $ \tilde y$ is the approximate eigenvector, and $\tilde r$ is the corresponding approximate eigenvalue. Note that our eigensolver consistently outperforms the Locally Optimal Preconditioned Conjugate Gradient Method with LAMG as a preconditioner.

\begin{table}[ht]
\caption{Numerical Tests} \label{tab:test}
\begin{center}
\begin{tabular}{ l | c | c | c | c | c | c }

\multicolumn{3}{c}{} & \multicolumn{2}{|c|}{LOPCG w/ LAMG} & \multicolumn{2}{c}{CMG Eigensolver} \\
Graph & Vertices & Edges & Run Time & Error & Run Time & Error  \\ \hline \hline
144 & 144649 & 1074393 & 6.254 & 5.0e-02 & 1.911 & 6.9e-03  \\
598a &  110971 & 741934 & 4.884 & 1.6e-02 & 1.414 & 6.8e-03 \\
auto & 448695 & 3314611& 13.34 & 3.9e-02 & 6.050 & 9.2e-03 \\
brack2 & 62631 & 366559 & 2.726 & 8.7e-03 & 0.780 & 8.3e-03 \\
cs4 & 22499 & 87716 & 1.194 & 2.0e-02 & 0.271 & 1.1e-03 \\
cti & 16840 & 96464 & 1.236 & 4.5e-02 & 0.283 & 1.7e-03 \\
delaunayn15 & 32768 & 196548 & 1.614 & 9.3e-03 & 0.389 & 4.8e-03 \\
m14b & 214765 & 3358036 & 9.210 & 2.4e-02 & 2.848 & 1.0e-02 \\
PGPgiantc. & 10680 & 48680 & 0.873 & 2.4e-02 & 0.201 & 5.8e-02 \\
wing & 62032 & 243088 & 2.232 & 1.2e-02 & 0.741 & 1.1e-03 \\
\end{tabular}
\end{center}
\end{table}

We consider the number of steps of power iteration that we typically require for a given graph. We use the two dimensional Laplacian with $N=10^3$ as an example. We implement our eigensolver, with our given tolerance and report the number of subgraphs we have, the size of each subgraph, and the number of iterations required on each level in Table \ref{tab:k_power}.

\begin{table}[ht]
\caption{Graph Size and Number of Smoothing Steps by Level for 2D Laplacian, $N=100$} \label{tab:k_power}
\begin{center}
\begin{tabular}{c | c | c }
$i$ & $n_i$ & $k_i$ \\ \hline
 0 & 10000 & 3 \\
1 & 3653 & 7 \\
2 & 1195 & 10 \\
3 & 384 & 17 \\
4 & 137 & 26 \\
5 & 46 & 44  \\
6 & 14 & - \\ 
\end{tabular}
\end{center}
\end{table}

We note that we observe a similar smoothing structure on each level to the condition $k_j = \beta^j k_0$ we assumed for the proof of Theorem \ref{maintheor}. Also, we note that the coarsening appears to occur at roughly the same rate on each level, suggesting that although our heavy edge coarsening algorithm is random in nature, it typically maintains similar coarsening rates for a given graph structure. These two observations help give numerical evidence that the theoretical results from Section \ref{sec:GCMG} are robust to general graphs with heavy edge coarsening as the restriction operator.

Finally, we give an example of the application of the GCMG algorithm to the two-dimensional Laplacian, to give some numerical results to support the theoretical bounds we obtained in Section \ref{sec:GCMG}. We choose $N=1025$ and take $\beta = 4$, $k_0=1$. We give the error with respect to the difference in eigenvalue, taking $\tilde r = \tilde y ^T L \tilde y$. Our results are given in Table \ref{tab:gcmg}.

\begin{table}[ht]
\caption{Errors on Sublevels for GCMG for 2D Laplacian, $N=1025$} \label{tab:gcmg}
\begin{center}
\begin{tabular}{c | c | c }
$i$ & $| \lambda^{(2)}_i - \tilde r_0 |$ & $| \lambda^{(2)}_i - \tilde r_{k_i} |$ \\ \hline
 0 & 3.0369e-09 & 3.0198e-09 \\
1 & 1.1189e-08 & 1.0826e-08 \\
2 & 4.0447e-08 & 3.4420e-08 \\
3 & 1.8336e-07 & 8.1745e-08 \\
4 & 1.8068e-06 & 8.2292e-08 \\
\end{tabular}
\end{center}
\end{table}

\section{Conclusion} \label{sec:conclusion}

In this paper, we have presented a fast algorithm for approximately computing the Fiedler vector of a graph Laplacian. We introduced a new coarsening procedure, called heavy edge coarsening. We note the speed with which the procedure coarsens, and the quality of coarse level graphs. The main contribution to the speed of the algorithm was a result of the implementation of the heavy edge coarsening procedure. 

In addition to being a fast coarsening procedure, the heavy edge coarsening algorithm is also easier to implement than other techniques of a similar type, such as heavy edge matching and its variants (HEM and HEM*) \cite{KK1,KK2}. As a purely algebraic eigensolver, the combination of heavy edge coarsening and power iteration in a cascadic multigrid method provide a fast algorithm for finding the Fiedler vector of graph Laplacians.  Numerical results show that our eigensolver is efficient and robust for different graphs.  

Similar to the AMG method, the algebraic CMG eigensolver is difficult to analyze.  Therefore,  under a standard geometric setting, we consider the GCMG eigensolver and show that our cascadic eigensolver with power iteration as a smoother to be uniformly convergent for an elliptic eigenvalue problem discretized by standard linear finite element methods.  In the three-dimensional case, it is optimal in terms of accuracy and computational complexity. 

We believe that in future work convergence for the cascadic multigrid eigensolver could be shown in more general settings. In addition, the use of the heavy edge coarsening procedure for non-spectral methods is another avenue of research that could be explored in the future.

\section*{Acknowledgement}
The research of Jinchao Xu is partially supported by NSF Grant DMS-1217142 and NSFC Grant 91130011/A0117. The research of Ludmil Zikatanov was supported in part by NSF grant DMS-1217142, and Lawrence Livermore National Laboratory through subcontract B603526. 

\bibliography{bib_Fiedler_paper.bib}

\end{document}